\documentclass[11pt]{amsart}

\usepackage[utf8]{inputenc}

\usepackage{amssymb}
\usepackage{upref}
\RequirePackage{doi}
\usepackage{hyperref}
\usepackage[capitalize]{cleveref}
\usepackage{mathtools}
\usepackage{tikz-cd}
\usepackage{graphicx, caption}
\usepackage{comment,enumerate}

\usepackage{adjustbox,lipsum}

\numberwithin{equation}{section}

\makeatletter \renewcommand{\p@enumii}{} \makeatother

\newtheorem{cor}[equation]{Corollary}
\newtheorem{lem}[equation]{Lemma}
\newtheorem{prop}[equation]{Proposition}
\newtheorem{thm}[equation]{Theorem}
\newtheorem{problem}[equation]{Problem}
\newtheorem{construction}[equation]{Construction}

\crefformat{lem}{Lemma~#2\textup{#1}#3}
\Crefname{lem}{Lemma}{Lemmas}
\crefname{lem}{lemma}{lemmas}
\crefmultiformat{lem}{Lemmas~#2\textup{#1}#3}{ and~#2\textup{#1}#3}{, #2\textup{#1}#3}{ and~#2\textup{#1}#3}
\crefformat{cor}{Corollary~#2\textup{#1}#3}
\Crefname{cor}{Corollary}{Corollaries}
\crefname{cor}{corollary}{corollaries}
\crefformat{prop}{Proposition~#2\textup{#1}#3}
\Crefname{prop}{Proposition}{Propositions}
\crefname{prop}{proposition}{propositions}
\crefmultiformat{prop}{Propositions~#2\textup{#1}#3}{ and~#2\textup{#1}#3}{, #2\textup{#1}#3}{ and~#2\textup{#1}#3}
\crefformat{thm}{Theorem~#2\textup{#1}#3}
\Crefname{thm}{Theorem}{Theorems}
\crefname{thm}{theorem}{theorems}

\theoremstyle{definition}

\newtheorem{defn}[equation]{Definition}
\newtheorem{eg}[equation]{Example}

\newtheorem*{notation*}{Notation}
\newtheorem{rem}[equation]{Remark}

\theoremstyle{remark}

\crefmultiformat{rem}{Remarks~#2\textup{#1}#3}{ and~#2\textup{#1}#3}{, #2\textup{#1}#3}{ and~#2\textup{#1}#3}
\crefformat{rem}{Remark~#2\textup{#1}#3}
\crefformat{rems}{Remark~#2\textup{#1}#3}
\Crefname{rem}{Remark}{Remarks}
\crefname{rem}{remark}{remarks}

\crefformat{figure}{Figure~#2\textup{#1}#3}

\newcounter{case}

\numberwithin{case}{equation}
\renewcommand{\thecase}{\arabic{case}}
\crefformat{case}{Case~#2\textup{#1}#3}
\Crefname{case}{Case}{Cases}
\crefname{case}{case}{cases}

\newcounter{subcase}

\numberwithin{subcase}{case}
\crefformat{subcase}{Subcase~#2\textup{#1}#3}
\crefname{subcase}{subcase}{subcases}
\Crefname{subcase}{Subcase}{Subcases}

\newcounter{subsubcase}

\numberwithin{subsubcase}{subcase}
\crefformat{subsubcase}{Subsubcase~#2\textup{#1}#3}
\crefname{subsubcase}{subsubcase}{subsubcases}
\Crefname{subsubcase}{Subsubcase}{Subsubcases}

\usepackage{color}
\theoremstyle{definition}
\newtheorem{aid}{}
\numberwithin{aid}{section}
\newcommand{\oldaid}{}
\let\oldaid=\aid
\renewcommand{\aid}{\bigbreak\oldaid}
\newcommand{\oldendaid}{}
\let\oldendaid=\endaid
\renewcommand{\endaid}{\oldendaid\bigskip\hrule width\textwidth \bigbreak}
\crefformat{aid}{Note~#2\textup{#1}#3}

\let\oldenumerate=\enumerate
\def\enumerate{\oldenumerate \itemsep=\smallskipamount}
\let\olditemize=\itemize
\def\itemize{\olditemize \itemsep=\smallskipamount}

\DeclareMathOperator{\Aut}{Aut}

\newcommand{\pref}[1]{\textup(\ref{#1}\textup)}

\makeatletter
\newcommand{\noprelistbreak}{\smallskip\@nobreaktrue\nopagebreak} 
\makeatother

\begin{document}

\title{Some conditions implying stability of graphs}

\author{Ademir Hujdurović}
\address{University of Primorska, UP IAM, Muzejski trg 2, 6000 Koper, Slovenia and University of Primorska, UP FAMNIT, Glagolja\v ska 8, 6000 Koper, Slovenia}
\email{ademir.hujdurovic@upr.si}

\author{Đorđe Mitrović}
\address{University of Primorska, UP FAMNIT, Glagolja\v ska 8, 6000 Koper, Slovenia}
\email{mitrovic98djordje@gmail.com}

\date{\today}

\begin{abstract}
   A graph $X$ is said to be {\em unstable} if the direct product $X\times K_2$ (also called the {\em canonical double cover} of $X$) has automorphisms that do not come from automorphisms of its factors $X$ and $K_2$. It is {\em non-trivially unstable} if it is unstable, connected, non-bipartite, and distinct vertices have distinct sets of neighbours.
   
   In this paper, we prove two sufficient conditions for stability of graphs in which every edge lies on a triangle, revising an incorrect claim of Surowski and filling in some gaps in the proof of another one. We also consider triangle-free graphs, and prove that there are no non-trivially unstable triangle-free graphs of diameter 2. An interesting construction of non-trivially unstable graphs is given and several open problems are posed.
\end{abstract}

\maketitle


\section{Introduction}

All graphs considered in this paper are finite, simple and undirected. For a graph $X$, we denote by $V(X)$, $E(X)$ and $\Aut(X)$ the vertex set, the edge set and the automorphism group of $X$, respectively. 
{\it The canonical bipartite double cover} (also called {\it the bipartite double cover} or {\it the Kronecker cover}) of a graph $X$, denoted by $BX$, is the direct product $X \times K_2$ (where $K_2$ denotes the complete graph on two vertices).  
This means that $V(BX)=V(X)\times \{0,1\}$ and $E(BX)=\{\{(x,0),(y,1)\}\mid \{x,y\}\in E(X)\}$. 

Canonical double covers have proven to play an important role in algebraic graph theory and have been studied by multiple groups of authors from a variety of perspectives \cite{
FengKutnar2fold,Hujdurovic,Krnc,Matsushita, 
Kneser2017,RegularEmbeddingsNedela,Pacco,WallerDouble,ZelinkaDouble}.
It is well-known that $BX$ is connected if and only if $X$ is connected and non-bipartite, see \cite{HIK}.
It is easy to see that $\Aut(BX)$ contains a subgroup isomorphic to $\Aut(X)\times S_2$.  However, determining the full automorphism group of $BX$ is not as trivial. Hammack and Imrich \cite{HI} investigated vertex-transitivity of the direct product of graphs, and proved that for a non-bipartite graph $X$ and a bipartite graph $Y$, their direct product $X\times Y$ is vertex-transitive if and only if both $BX$ and $Y$ are vertex-transitive. Hence, the problem of vertex-transitivity of the direct product of graphs reduces to the problem of vertex-transitivity of canonical double covers.

If $\Aut(BX)$ is isomorphic to $\Aut(X)\times S_2$ then the graph $X$ is called {\it stable}, otherwise it is called {\it unstable}. 
This concept was first defined by Maru\v si\v c et al. \cite{maruvsivc1989characterization}. A graph is said to be {\em twin-free} (also called {\em worthy}, {\em vertex-determining} or {\em irreducible}) if distinct vertices have different neighbourhoods.
It is not difficult to prove that disconnected graphs, bipartite graphs with non-trivial automorphism groups, and graphs containing twin vertices are unstable. These graphs are considered {\em trivially unstable}. An unstable graph is said to be {\em non-trivially unstable} if it is non-bipartite, connected, and twin-free. 

In \cite{WilsonUnExpected}, Wilson gave several sufficient conditions for a graph to be unstable, which he then applied to families of circulants, Rose Window graphs, toroidal graphs and generalized Petersen graphs. The characterization of unstable generalized Petersen graphs was obtained by Qin, Xia and Zhou in \cite{Qin21}. Stability of circulant graphs was studied by Qin, Xia and Zhou in \cite{QinXiaZhou}, where it was proven that there are no non-trivially unstable arc-transitive circulants. In \cite{FernandezHujdurovicCirc}, it was proven by Fernandez and Hujdurovi\' c that there are no non-trivially unstable circulants of odd order. In \cite{morris2021automorphisms}, Morris extended this result by proving that there are no non-trivially unstable Cayley graphs on Abelian groups of odd order. In \cite{HMMAtMost7}, the complete classification of unstable circulants of valency at most $7$ was obtained, while in \cite{HMMAuto} several constructions of unstable circulants were presented, and unstable circlants of order $2p$ (with $p$ a prime) were classified.

Although it is in some sense expected that most graphs are stable, there are not many results giving sufficient conditions for a general graph to be stable. 
The main motivation for this paper comes from the work of Surowski \cite{SurowskiStabArcTrans} on stability of arc-transitive graphs, where he describes two sufficient stability conditions, namely {\cite[Proposition~2.1]{SurowskiStabArcTrans}} for vertex-transitive graphs of diameter at least 4 and {\cite[Proposition~2.2]{SurowskiStabArcTrans}} for strongly regular graphs. 
The first of these results has been shown not to hold by Lauri, Mizzi and Scapellato in \cite{LauriMizziScapellato}, where the authors constructed an infinite family of counterexamples. In the same article, the authors pointed out that the proof of the second result seems incomplete, hence asking for a further investigation of stability of strongly regular graphs.

In \cref{SectionTriangles}, we will discuss the original statement of Surowski's first stability criterion and an infinite family of counterexamples by Lauri, Mizzi and Scapellato (see \cref{LMScounterexample}). Then we will explain the mistake in the original proof and fix it by introducing additional assumptions, consequently obtaining a valid stability criterion in \cref{SurowskiDiameterUpdated}, which no longer requires the graph to be vertex-transitive. Then we turn our attention to the second result of Surowski concerning strongly regular graphs, which turned out to be correct, although its original proof was slightly incomplete. We prove a generalized version of this result in \cref{StableTriangleEdgeNonEdge} that applies to a wider class of graphs. 

In \cref{SectionTriangleFree}, to contrast the previous results which required every edge of a graph to lie on a triangle, we consider triangle-free graphs, and prove that there are no non-trivially unstable triangle-free graphs of diameter $2$ (see  \cref{TriangleFreeStable}).

In \cref{sec:construction}, we present an interesting construction of non-trivially unstable graphs, which shows that every connected, non-bipartite, twin-free graph of order $n$ is an induced subgraph of a non-trivially unstable graph of order $n+4$. With the help of a computer, we check that most of the non-trivially unstable graphs up to 10 vertices can be obtained using this construction.

\section{Preliminaries} \label{sec:preliminaries}
Let $X$ be a graph and $BX$ its bipartite double cover. For an automorphism $\varphi$  of $X$, it is easy to see that the function   $\overline{\varphi}$, defined by $\overline{\varphi}(x,i)=(\varphi(x),i)$ is an automorphism of $BX$, called \textit{the lift of $\varphi$}.
It is straightforward to check that the function $\tau$ defined by $\tau(x,i)=(x,i+1)$ (the second coordinate is calculated modulo 2) is also an automorphism of $BX$.

The subgroup generated by $\tau$ and the lifts $\overline{\varphi}$ with $\varphi\in\Aut(X)$ is isomorphic to $\Aut(X)\times S_2$. Automorphisms of $BX$ that belong to this subgroup are called the {\em expected automorphisms} of $BX$ (see \cite{WilsonUnExpected}). If an automorphism of $BX$ does not belong to this subgroup, it is called {\em unexpected}.
Using this terminology, a graph $X$ is stable if and only if every automorphism of $BX$ is expected.

The following result is known and will be used later on for establishing stability of graphs. For the sake of completeness, we provide a short proof.

\begin{lem}\label{lem:stability (x,0)(x,1)}
Let $X$ be a connected, non-bipartite graph, and let $\alpha$ be an automorphism of $BX$. Then $\alpha$ is an expected automorphism of $BX$ if and only if for every $x\in V(X)$, it holds that $\alpha(\{(x,0),(x,1)\})=\{(y,0),(y,1)\}$ for some $y\in V(X)$.
\end{lem}
\begin{proof}
Since $X$ is a connected and non-bipartite graph, it follows that $BX$ is connected and bipartite with bipartite sets $V(X)\times \{0\}$ and $V(X)\times \{1\}$. If $\alpha$ interchanges the bipartite sets, then one can consider $\alpha\tau$ instead. Hence, without loss of generality, we can assume that $\alpha$ preserves the sets $V(X)\times \{0\}$ and $V(X)\times \{1\}$. Then $\alpha$ satisfies the condition from Lemma \ref{lem:stability (x,0)(x,1)} if and only if $\alpha(x,i)=(\varphi(x),i)$, for some permutation $\varphi$ of $V(BX)$. It is easy to see that the map $(x,i)\mapsto (\varphi(x),i)$ is an automorphism of $BX$ if and only if $\varphi$ is an automorphism of $X$.
\end{proof}

It is not difficult to show that even cycles are unstable, while odd cycles are stable. Moreover, we have the following easy example.

\begin{eg}[Qin-Xia-Zhou,{\cite[Example~2.2]{QinXiaZhou}}]\label{CompStab}
The complete graph $K_n$ is unstable if and only if $n = 2$.
\end{eg}

Recall that a {\em  strongly regular graph} with parameters $(n,k,\lambda,\mu)$ is a $k$-regular connected graph of diameter $2$ on $n$ vertices such that every edge lies on $\lambda$ triangles, and any two non-adjacent vertices have exactly $\mu$ neighbours in common. Note that $\mu>0$ and that a strongly regular graph is twin-free if and only if $k>\mu$.

For a graph $X$ and a vertex $x$ of $X$, we will denote by $X_i(x)$ the set of vertices of $X$ at distance $i$ from $x$. We will also sometimes write $N_X(x)$ for the set of neighbours of $x$, instead of $X_1(x)$.

\section{Stability of graphs with many triangles}\label{SectionTriangles}

In their article \cite{LauriMizziScapellato} on TF-automorphisms of graphs, Lauri, Mizzi and Scapellato describe a method for constructing unstable graphs of an arbitrarily large diameter with the property that every edge lies on a triangle.
This contradicts {\cite[Proposition~2.1]{SurowskiStabArcTrans}}, a result of Surowski claiming that a connected, vertex-transitive graph of diameter at least $4$ with every edge on a triangle is stable. The particular family of counterexamples described by Lauri, Mizzi and Scapellato (based on their result {\cite[Theorem~5.1]{LauriMizziScapellato}}) can be reformulated in terms of the  lexicographic product of graphs as follows (see \cite[Definition 4.2.1.]{DobsonBook} for the definition of the lexicographic product of graphs).

\begin{prop}\label{LMScounterexample}
Let $H$ be a non-trivial vertex-transitive graph that is also twin-free and bipartite. Let $m\geq 8$ be a positive integer. Then $C_m\wr H$ is a connected, non-trivially unstable, vertex-transitive graph of diameter at least $4$ whose every edge lies on a triangle.
\end{prop}

A particular example pointed out by Lauri, Mizzi and Scapellato at the end of their article is the graph $C_8\wr C_6$. In the language of \cref{LMScounterexample}, we set $m=8$ and $H=C_6$.

In his proof, Surowski uses the assumption that every edge lies on a triangle to describe a part of the distance partition of $BX$ with respect to a fixed vertex $(x,0)$. In particular, Surowski shows that the vertex $(x,1)$ is at distance $3$ from $(x,0)$. The author then attempts to show that $(x,1)$ is the unique vertex at distance $3$ from $(x,0)$ in $BX$ with no neighbours lying at distance $4$ from $(x,0)$. This would imply that any automorphism of $BX$ fixing $(x,0)$ must also fix $(x,1)$, which is equivalent to $X$ being stable as it is assumed to be vertex-transitive (see for example \cite[Lemma~2.4]{FernandezHujdurovicCirc}).

However, the last claim about $(x,1)$ does not always hold. For example, if $X=C_8\wr C_6$ and $x$ is its arbitrary vertex, then any of the $3$ vertices lying in the same copy of $C_6$ as $x$, that are distinct from $x$ and are not adjacent to it, induce vertices in $BX$ at distance $3$ from $(x,0)$, all of whose neighbours are at distance at most $2$ from $(x,0)$. In particular, it is not possible to distinguish $(x,1)$ from other elements of $(BX)_3(x,0)$ in the manner proposed by Surowski.

We remedy this situation by adding additional assumptions on $X$. On the other hand, we no longer require $X$ to be vertex-transitive.

\begin{prop}\label{SurowskiDiameterUpdated}
Let $X$ be a non-trivial connected graph. Assume that $X$ satisfies the following conditions.

\begin{enumerate}
    \item \label{SurowskiDiameterUpdated-triangle} Every edge of $X$ lies on a triangle.
    \item \label{SurowskiDiameterUpdated-distance} For every $x\in V(X)$, it holds that
        \begin{enumerate}
            \item \label{SurowskiDiameterUpdated-distance-2} every vertex at distance $2$ from $x$ has a neighbour at distance $3$ from $x$, and
            \item \label{SurowskiDiameterUpdated-distance-3} every vertex at distance $3$ from $x$ has a neighbour at distance $4$ from $x$.
        \end{enumerate}
\end{enumerate}

Then $X$ is stable.
\end{prop}

\begin{proof}
It is easy to show that if $X_2(x)$ is empty for some $x\in V(X)$, and $X$ satisfies the assumptions, then $X$ is complete and therefore, stable by \cref{CompStab} (assumption \pref{SurowskiDiameterUpdated-triangle} rules out the case $X=K_2$). We can therefore assume that $X_2(x)$ is non-empty for all $x\in V(X)$. Conditions \pref{SurowskiDiameterUpdated-distance}\pref{SurowskiDiameterUpdated-distance-2} and \pref{SurowskiDiameterUpdated-distance}\pref{SurowskiDiameterUpdated-distance-3} then imply that also $X_3(x)$ and $X_4(x)$ are non-empty for all $x\in V(X)$.

Observe that
\[(BX)_2(x,0)=X_1(x)\times \{0\}\cup X_2(x)\times \{0\},\]
and
\[(x,1)\in (BX)_3(x,0)\subseteq \{(x,1)\}\cup X_2(x)\times \{1\}\cup X_3(x)\times \{1\}.\]

From here, it is clear that all neighbours of $(x,1)$ in $BX$ lie in $(BX)_2(x,0)$. Consequently, $(x,1)$ has no neighbours in $(BX)_4(x,0)$. We will show that $(x,1)$ is the unique element of $(BX)_3(x,0)$ with this property. To show this, we observe the following.
\begin{itemize}
    \item If $y\in X_2(x)$, then by assumption \pref{SurowskiDiameterUpdated-distance}\pref{SurowskiDiameterUpdated-distance-2}, it has a neighbour $z\in V(X)$ such that $z\in X_3(x)$. Then $(z,0)\in (BX)_4(x,0)$ is a neighbour of $(y,1)$.
    \item If $y\in X_3(x)$, then by assumption \pref{SurowskiDiameterUpdated-distance}\pref{SurowskiDiameterUpdated-distance-3}, it has a neighbour ~$z\in V(X)$ such that $z\in X_4(x)$. By the same arguments as before, it follows that $(z,0)\in (BX)_4(x,0)$ is a neighbour of $(y,1)$.
\end{itemize}

Let $\alpha\in\Aut(BX)$ be arbitrary. As $X$ is connected and non-bipartite (as it contains triangles), after composing $\alpha$ with $\tau$ if necessary, we can assume that
\[\text{$\alpha(V(X)\times \{i\})=V(X)\times \{i\}$ for $i\in \{0,1\}$.}\]
Let $x\in V(X)$ be arbitrary and choose $y\in V(X)$ such that $\alpha(x,0)=(y,0)$. As $\alpha$ is an automorphism, it preserves distances, so it follows that
\[\alpha(x,1)\in (BX)_3(\alpha(x,0)) = (BX)_3(y,0).\]

As $(x,1)$ has no neighbours in $(BX)_4(x,0)$, it follows that $\alpha(x,1)$ has no neighbours in $(BX)_4(y,0)$. As we have already observed, the only element in $(BX)_3(y,0)$ with this property is $(y,1)$. It follows that $\alpha(x,1)=(y,1)$. The stability of $X$ now follows by Lemma \ref{lem:stability (x,0)(x,1)}.
\end{proof}

Our updated criterion has a nice application to distance-regular graphs.
Recall that a regular graph $X$ of diameter $d$ is said to be {\em distance-regular} if there exists an array of integers $\{b_0,\ldots,b_{d-1},c_1,\ldots,c_d\}$, called the {\em intersection array}, with the property that for all $1\leq j\leq d$, it holds that for any pair of vertices $x,y\in V(X)$ at distance $j$ in $X$, the number of neighbours of $y$ at distance $j+1$ from $x$ is $b_j$, while the number of neighbours of $y$ at distance $j-1$ from $x$ is $c_j$.

\begin{cor}\label{StableDistanceReg}
Let $X$ be a connected distance-regular graph of diameter $d$ with an intersection-array $\{b_0,\ldots,b_{d-1},c_1,\ldots,c_d\}$. If $d\geq 4$, $b_0>b_1+1$ and $b_2,b_3\geq 1$, then $X$ is stable.
\end{cor}
\begin{proof}
The assumption that $b_0>b_1+1$ guarantees that every edge of $X$ lies on a triangle, while the assumption that $b_2,b_3\geq 1$ guarantees that the conditions \pref{SurowskiDiameterUpdated-distance}\pref{SurowskiDiameterUpdated-distance-2} and \pref{SurowskiDiameterUpdated-distance}\pref{SurowskiDiameterUpdated-distance-3} from \cref{SurowskiDiameterUpdated} are satisfied.
\end{proof}

We now turn our attention to the following result of Surowski {\cite[Proposition~2.2]{SurowskiStabArcTrans}}.
\begin{prop}[Surowski {\cite[Proposition~2.2]{SurowskiStabArcTrans}}]\label{StableSRGLambdaMu}
Let $X$ be a strongly regular graph with parameters $(n,k,\lambda,\mu)$. If $k>\mu \neq \lambda \geq 1$, then $X$ is stable.
\end{prop}

Lauri, Mizzi and Scapellato  observed in \cite{LauriMizziScapellato} that the proof given by Surowski is not complete, and they suggested \cref{StableSRGLambdaMu} should be reviewed. Surowski's original proof implies that $\Aut(BX)_{(x,0)} = \Aut(BX)_{(x,1)}$ for all $x\in V(X)$. This is a necessary, but in general not sufficient condition for a graph to be stable. For example, the Swift graph shown in \cref{SGGraph} satisfies this condition, but is non-trivially unstable (see \cite{WilsonUnExpected} for more details on the Swift graph). We will show that \cref{StableSRGLambdaMu} is correct, by proving a more general result.

\begin{figure}[!htbp]
    \centering
    \includegraphics[width = 0.75\linewidth]{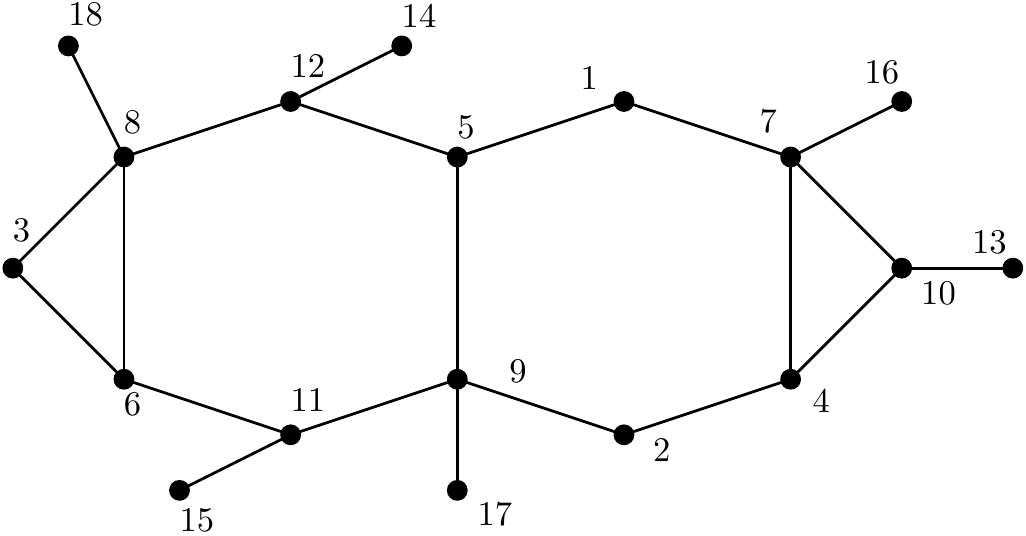}
    \caption{The Swift graph $SG$ {\cite[Figure~17.]{WilsonUnExpected}}$\quad$}
    \label{SGGraph}
\end{figure}

\begin{prop}\label{StableTriangleEdgeNonEdge}
Let $X$ be a non-trivial connected, twin-free graph such that every edge of $X$ lies on a triangle.
If for all pairs $x,y\in V(X)$ of adjacent vertices and all pairs $z,w\in V(X)$ of vertices at distance $2$ from each other, it holds that
\begin{equation*}
    |N_X(x)\cap N_X(y)| \neq |N_X(z)\cap N_X(w)|,
\end{equation*}
then $X$ is stable.
\end{prop}

\begin{proof}
As every edge of $X$ lies on a triangle, for every $x\in V(X)$ it holds that
\[(BX)_2(x,0) = X_1(x)\times \{0\}\cup X_2(x)\times \{0\}.\]
  
Let $\alpha\in \Aut(BX)$ be arbitrary. As $X$ is connected and non-bipartite (as it contains triangles), after possibly composing $\alpha$ with $\tau$, we may assume that $\alpha(V(X)\times \{i\})=V(X)\times\{i\}$ for $i\in \{0,1\}$. Let $x\in V(X)$ and choose $z\in V(X)$ such that $\alpha(x,0)=(z,0)$.

Let $y\in X_1(x)$ and choose $w\in V(X)$ such that $\alpha(y,0)=(w,0)$. As $\alpha$ is an automorphism of $BX$, it preserves distance, so since $(y,0)\in (BX)_2(x,0)$, we have that
\[(w,0) \in (BX)_2(\alpha(x,0)) =(BX)_2(z,0) = X_1(z)\times \{0\}\cup X_2(z)\times \{0\}.\]

The number of neighbours of $(w,0)$ lying in $(BX)_1(z,0)$ is by definition $|N_X(z)\cap N_X(w)|$, but as $\alpha$ is an automorphism of $BX$, this number also equals to $|N_X(x)\cap N_X(y)|$. As $\{x,y\}$ is an edge of $X$, if $w\in X_2(z)$, we would arrive at a contradiction with our assumption. Therefore, $w\in X_1(z)$ and $\alpha(y,0)\in X_1(z)\times \{0\}$.
We conclude that 
\[\alpha(X_1(x)\times \{0\})\subseteq X_1(z)\times \{0\}.\]

As $\alpha$ is an automorphism of $BX$, it follows that $(x,0),(z,0),(x,1)$ and $(z,1)$ are all of the same valency, so we have in fact proven that
\[N_{BX}(\alpha(x,1))=\alpha(N_{BX}(x,1))=\alpha(X_1(x)\times \{0\})= X_1(z)\times \{0\} = N_{BX}(z,1).\]

Hence, $\alpha(x,1)$ and $(z,1)$ are twins in $BX$, and as $X$ is twin-free, we conclude that $\alpha(x,1)=(z,1)$. By Lemma~\ref{lem:stability (x,0)(x,1)}, it follows that $X$ is stable.
\end{proof}

We can now obtain the original result of Surowski as a corollary of the one we just established. 

\begin{proof}[Proof of \cref{StableSRGLambdaMu}]
First we note that $X$ must be connected, since vertices in distinct connected components would have no neighbours in common, which contradicts the fact that $\mu > 0$. Next, $X$ is twin-free, as twins are non-adjacent vertices that would have $k$ neighbours in common, which contradicts the assumption that $k>\mu$. Every edge of $X$ lies on a triangle since $\lambda\geq 1$.

Finally, let $x,y,z,w\in V(X)$ be such that $\{x,y\}\in E(X)$, and $z$ and $w$ lie at distance $2$. Then $z$ and $w$ are not adjacent and we have that
\[|N_X(x)\cap N_X(y)| = \lambda \neq \mu = |N_X(z)\cap N_X(w)|.\]

It follows that $X$ is stable by \cref{StableTriangleEdgeNonEdge}.
\end{proof}

The results we established can be used to analyze the stability of various families of graphs. Below, we will show how they can be used to give a complete classification of unstable Johnson graphs.

\begin{defn}[{\hspace{1sp}\cite[p.~9]{GodsilRoyleAlgGraphTh}}]\label{JohnsonDefn}
Let $n\geq k\geq 1$ be positive integers. The \textit{Johnson graph} is the graph $J(n,k)$ with $k$-element subsets of $\{1,\ldots,n\}$ as vertices, which are adjacent if and only if the size of their intersection as sets is $k-1$.  
\end{defn}

We recall that the Johnson graph $J(n,k)$ is connected with diameter $\min(k,n-k)$. Moreover, the map assigning to each subset of $\{1,\ldots,n\}$ its complement induces a graph isomorphism $J(n,k)\cong J(n,n-k)$. 

\begin{thm}\label{JohnsonStable}
Let $n\geq k\geq 1$ be positive integers. The Johnson graph $J(n,k)$ is unstable if and only if it is one of the following:
\begin{enumerate}
    \item \label{JohnsonStable-(2,1)} the complete graph $J(2,1)\cong K_2$,
    \item \label{JohnsonStable-(4,2)} the octahedral graph $J(4,2)$,
    \item \label{JohnsonStable-(6,2)(6,4)} $J(6,2)\cong J(6,4)$ or
    \item \label{JohnsonStable-(6,3)} $J(6,3)$.
\end{enumerate}
Moreover, $J(2,1)$ is bipartite, while $J(4,2)$ is not twin-free, so both are trivially unstable.~Graphs ~$J(6,2)$ and~$J(6,3)$ are non-trivially unstable with indices of instability $28$ and $2$, respectively.
\end{thm}
\begin{proof}
It is easy to see that the only non-trivial bipartite Johnson graph is $J(2,1)$, while the only Johnson graph admitting twins is $J(4,2)$. Hence, we assume that $(n,k)\not \in \{(2,1),(4,2)\}$.

It can be checked that any two adjacent vertices in $J(n,k)$ have $n-2$ neighbours in common, while any two vertices at distance $2$ have $4$ neighbours in common.

Applying \cref{StableTriangleEdgeNonEdge}, we conclude that $J(n,k)$ is stable whenever $n\neq 6$. Letting $n=6$, we see that $J(6,1)\cong J(6,5)\cong K_6$ is stable by \cref{CompStab}, while it can be verified with the help of a computer (for example using Magma \cite{MAGMA}) that $J(6,2)\cong J(6,4)$ and $J(6,3)$ are unstable with indicated indices of instability.
\end{proof}

\begin{rem}
Stability of Johnson graphs has been studied by Mirafzal in \cite{Mirafzal}, where it has been incorrectly claimed that all Johnson graphs are stable (see {\cite[Theorem~3.20]{Mirafzal}}).
\end{rem}

\section{A stability condition for triangle-free graphs}\label{SectionTriangleFree}

In the previous section, we considered stability of graphs containing many triangles. On the other end of the spectrum, we consider a stability criterion for graphs that do not contain any triangles whatsoever.

\begin{thm}\label{TriangleFreeStable}
    Let $X$ be a connected, non-bipartite and twin-free graph. If $X$ is triangle-free of diameter $2$, then $X$ is stable.
\end{thm}
\begin{proof}
Let $x\in V(X)$ be arbitrary and define
\[S(x)\coloneqq \{z\in X_2(x)\mid X_1(z)\cap X_2(x)\neq \emptyset\}.\]

We first establish the following
\[\text{$(x,1)\in (BX)_5(x,0)\subseteq \{(x,1)\}\cup(X_2(x)\setminus S(x))\times \{1\}.$}\]

Clearly, $(x,1)\not\in (BX)_1(x,0)$ and as $X$ is triangle-free, it also holds that $(x,1)\not\in (BX)_3(x,0)$. Note that as $X$ is non-bipartite, $S(x)$ is not empty (otherwise, $X_1(x)\cup (X_2(x)\cup \{x\})$ would be a bipartition of $X$). Hence, we obtain a path $x\sim y\sim z \sim w$ with $y\in X_1(x)$ and $z,w\in X_2(x)$. We know that $w$ must have a neighbour $v\in X_1(x)$, which is distinct from $y$ as $X$ is triangle-free. Hence, we obtained a $5$-cycle formed by the vertices $x,y,z,w$ and $v$, which implies that $(x,1)\in (BX)_5(x,0)$.

Let $(w,1)\in (BX)_5(x,0)$. We may assume that $w\neq x$. Note that $w\not\in X_1(x)$ as then $(w,1)\in (BX)_1(x,0)$. It follows that $w\in X_2(x)$. If $w$ had a neighbour $z\in X_2(x)$, then using an arbitrary $y\in X_1(x)\cap X_1(z)$, we would be able to construct a path $(x,0)\sim (y,1)\sim (z,0)\sim (w,1)$, showing that $(w,1)$ is at distance at most $3$ from $(x,0)$. It follows that $w\in X_2(x)\setminus S(x)$. This finishes the proof of the claim.

Let $y\in X_2(x)\setminus S(x)$. Then all neighbours of $y$ are contained in $X_1(x)$, by definition of $S(x)$. Since $X$ is twin-free, it follows that $y$ has smaller valency than $x$. Consequently $(y,1)$ has smaller valency than $(x,1)$. We conclude that
\[\text{$(x,1)$ is the unique element of $(BX)_5(x,0)$ of the same valency as $(x,0)$}.\]

Let $\alpha\in \Aut(BX)$ be arbitrary. After possibly composing $\alpha$ with $\tau$, we may assume that it preserves the colour classes of $BX$. Let $\alpha(x,0)=(y,0)$ for $x,y\in V(X)$.

As $\alpha$ preserves distances in $BX$, it follows that
\[\alpha(x,1)\in (BX)_5(\alpha(x,0))= (BX)_5(y,0).\]

Note that the valency of $\alpha(x,1)$ is equal to the valency of $(y,0)$. We conclude that $\alpha(x,1)=(y,1)$ and by Lemma~\ref{lem:stability (x,0)(x,1)} it follows that $X$ is stable.
\end{proof}

This result can be used to slightly extend Surowski's result for strongly regular graphs (see \cref{StableSRGLambdaMu}).

\begin{cor}\label{SRGLambda=0}
Let $X$ be an $(n,k,\lambda,\mu)$-strongly regular graph. If $k > \mu $ and $\lambda = 0$, then $X$ is stable.
\end{cor}
\begin{proof}

We make the following observations.
\begin{itemize}
    \item As $X$ is a strongly regular graph, it is connected with diameter $2$ by definition.
    \item $k>\mu$ implies that $X$ is twin-free.
    \item $\lambda = 0$ implies that $X$ is triangle-free.
    \item Every edge lies on a $5$-cycle, proving that $X$ is non-bipartite.
\end{itemize}

We can now apply \cref{TriangleFreeStable} to conclude that $X$ is stable.
\end{proof}

It is worth noting that besides the trivial example given by $K_{n,n}$ with $n\geq 2$, which is $(2n,n,0,n)$-srg, there are only seven other currently known strongly regular graphs that are triangle-free, as explained by Godsil in {\cite{AlgCombProblemsGODSIL}}.

As a corollary of the results we derived so for, we obtain the following stability criterion for strongly regular graphs.

\begin{prop}\label{SRGStable}
Let $X$ be an $(n,k,\lambda,\mu)$-strongly regular graph. If $X$ is non-trivially unstable, then $\lambda = \mu > 0$.
\end{prop}

\begin{proof}
As $X$ is non-trivially unstable, it is connected and twin-free, so $k>\mu>0$. As $X$ is unstable, \cref{SRGLambda=0} shows that $\lambda >0$. Finally, it follows by \cref{StableSRGLambdaMu} that $\mu = \lambda$, as desired.
\end{proof}

\begin{rem}
In \cite{SurowskiStabArcTrans}, Surowski constructs an infinite family of non-trivially unstable strongly regular graphs with $\lambda =\mu$.
\end{rem}

\begin{problem}
Characterize non-trivially unstable strongly regular graphs.
\end{problem}

\section{The number of unstable graphs}
\label{sec:construction}

An interesting question to consider is how dense is the set of all non-trivially unstable graphs in the set of all connected, non-bipartite and twin-free graphs. We will now present an interesting construction of non-trivially unstable graphs which shows that every connected, non-bipartite, twin-free graph of order $n$ is an induced subgraph of a non-trivially unstable graph of order $n+4$.

\begin{construction}\label{constructionX(A,B)}
Let $X$ be a graph and $A$ and $B$ be subsets of $V(X)$. Let $X(A,B)$ denote the graph with
\begin{itemize}
    \item $V(X(A,B))\coloneqq V(X)\cup \{a_1,a_2,b_1,b_2\}$, where $a_1,a_2,b_1$ and $b_2$ are four distinct vertices with the property that $a_1,a_2,b_1,b_2\not\in V(X)$,
    \item $E(X(A,B))\coloneqq E(X)\cup\{\{a_1,b_1\},\{a_2,b_2\}\}\cup  \{\{a_1,a\},\{a_2,a\}\mid a\in A\}\cup \{\{b_1,b\},\{b_2,b\}\mid b\in B\}$.
\end{itemize}
In particular, $X(A,B)$ is obtained from $X$ by adding two new edges $\{a_1,b_1\}$ and $\{a_2,b_2\}$ to $X$, and then joining $a_1$ and $a_2$ with every vertex in $A$ and $b_1$ and $b_2$ with every vertex in $B$ (see Figure~\ref{fig:construction}).
\end{construction}

\begin{figure}[!htbp]
    \centering
    \includegraphics{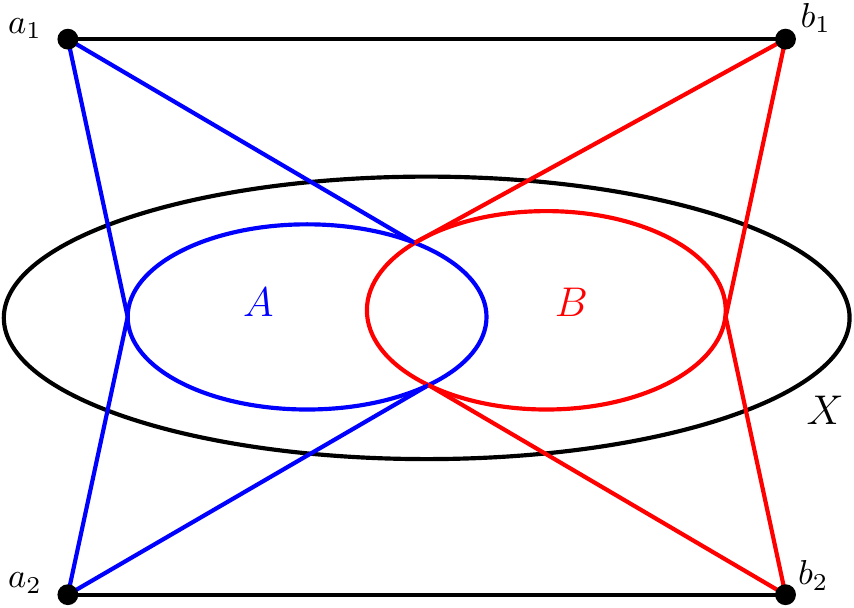}
    \caption{The construction of $X(A,B)$ from a graph $X$.}
    \label{fig:construction}
\end{figure}

\begin{prop}\label{WX}
Let $X$ be a graph and let $A$ and $B$ be subsets of $V(X)$. Then $X(A,B)$ is unstable.  Moreover, if $X$ is connected, non-bipartite and twin-free, and at least one of $A$ and $B$ is non-empty, then $X(A,B)$ is non-trivially unstable.
\end{prop}
\begin{proof}
It is easy to check that the permutation
\[\gamma^*\coloneqq ((a_1,0),(a_2,0))((b_1,1),(b_2,1))\]
of the vertex set of  $B(X(A,B))$ swapping $(a_1,0)$ with $(a_2,0)$ and $(b_1,1)$ with $(b_2,1)$, while fixing all other vertices, is an unexpected automorphism of $B(X(A,B))$ of order $2$, showing that $X(A,B)$ is unstable. 

Suppose now that $X$ is connected, non-bipartite and twin-free, and that $A$ is a non-empty set of vertices of $X$. It is clear that $X(A,B)$ is connected. Since $X$ is non-bipartite, and $X$ is an induced subgraph of $X(A,B)$ it follows that $X(A,B)$ is also non-bipartite. Observe that it follows from the definition that vertices $a_1,a_2,b_1,b_2$ have different sets of neighbours. Since no two vertices of $X$ have the same sets of neighbours, it follows that $X(A,B)$ is twin-free, showing that $X(A,B)$ is non-trivially unstable.
\end{proof}

In Table~\ref{tab:number of graphs}, for a positive integer $n$ between $3$ and $10$, the number given in the second row is the number of connected, non-bipartite, twin-free graphs of order $n$, the number in the third row is the number of non-trivially unstable graphs of order $n$, and the number in the fourth row is the number of non-trivially unstable graphs that can be realized using Construction~\ref{constructionX(A,B)}. The entries of Table~\ref{tab:number of graphs} have been obtained using Magma \cite{MAGMA}.

\begin{table}[h]
     \captionsetup{width=\linewidth}
\caption{\label{tab:number of graphs} The number of non-trivially unstable graphs up to 10 vertices.}
\begin{adjustbox}{max width=\textwidth}
\begin{tabular}{|c|c|c|c|c|c|c|c|c|}
\hline
    $n$ & 3 & 4 & 5 & 6 & 7 & 8 & 9 & 10 \\
    \hline
   Connected, non-bipartite, twin-free  & 1 & 2 & 10 & 56 & 498 & 7.397 & 197.612  & 9.807.191 \\
   \hline
   Non-trivially unstable & 0 & 0 & 1 & 6 & 43 & 395 &  5.113 & 105.919\\
   \hline
   $X(A,B)$ & 0 & 0 & 1 & 5 & 37 & 330 & 4.374 &  93.610 \\
   \hline
\end{tabular}
\end{adjustbox}
\end{table}

The data in Table~\ref{tab:number of graphs} suggests that the density of non-trivially unstable graphs among all connected, non-bipartite, twin-free graphs goes to $0$ as $n$ tends to infinity. On the other hand, it seems that almost all non-trivially unstable graphs can be constructed using Construction~\ref{constructionX(A,B)}. 


We conclude the paper with the following open problem.
\begin{problem}
Give an approximation formula for the number of non-trivially unstable graphs of order $n$.
\end{problem}

\section*{Acknowledgements}

The work of Ademir Hujdurovi\'c  is supported in part by the Slovenian Research Agency (research program P1-0404 and research projects J1-1691, J1-1694, J1-1695, N1-0102, N1-0140, N1-0159, J1-2451, N1-0208 and J1-4084).

\bibliography{References}
\bibliographystyle{abbrv}
\end{document}